\documentclass{article}
\usepackage{amsmath,amsfonts,amsthm,amssymb}
\DeclareMathSymbol\nullset{\mathord}{AMSb}{"3F}

\begin{document}
\newtheorem{theorem}{Theorem}[section]
\newtheorem{corollary}{Corollary}[section]
\newtheorem{lemma}{Lemma}[section]
\newtheorem{proposition}{Proposition}[section]
{\newtheorem{definition}{Definition}[section]}
\newcounter{sarparts}
\title{\centerline{\bf On a conjecture of V. V. Shchigolev}}
\author{C. Bekh-Ochir and S. A. Rankin
% \address{Department of Mathematics, University of Western Ontario\\ 
%         London, Ontario N5A 2B1, Canada}
  %         \email{cbekhoch@uwo.ca}
  %         \email{srankin@uwo.ca}
  }
%%MSC 16R10
\maketitle
%%%%%%%%%%%%%%%%%%%%%%%555
\newcounter{parts}
\def\com#1,#2{[\,{#1},{#2}\,]}
\def\kx{k\langle X\rangle}
\def\kzerox{k_0\langle X\rangle}
\def\konex{k_1\langle X\rangle}
\def\set#1\endset{\{\,#1\,\}}
\def\rest#1{\,\hbox{\vrule height 6pt depth 9pt width .5pt}_{\,#1}}         
\def\choice#1,#2{\binom{#1}{#2}}
\def\kzerox{k_0\langle X\rangle}
\let\cong=\equiv
\def\comp{\mkern2mu\mathchoice%
        {\raise.35ex\hbox{$\scriptscriptstyle\circ$}}
        {\raise.35ex\hbox{$\scriptscriptstyle\circ$}}
        {\raise.14ex\hbox{$\scriptscriptstyle\circ$}}
        {\raise.14ex\hbox{$\scriptscriptstyle\circ$}}}
\def\from{\mkern2mu\hbox{\rm :}\mkern2mu}
\def\ns#1#2{S_{#1}^{(#2)}}
\def\nsp#1{\ns#1{p}}
\def\rd#1#2{R_{{#1}}^{(#2)}}
\def\rdp#1{\rd{#1}{p}}
\def\nh#1{H_{#1}}
\def\nl#1{L_{#1}}

\begin{abstract}
 V. V. Shchigolev has proven that over any infinite field $k$ of characteristic
 $p>2$, the $T$-space generated by $G=\{\,x_1^p,x_1^px_2^p,\ldots\,\}$ is finitely based, which answered a question raised
 by A. V. Grishin. Shchigolev went on to conjecture that every infinite subset of $G$ generates
 a finitely based $T$-space. In this paper, we prove that Shchigolev's conjecture was correct by showing that for 
 {\em any} field of characteristic $p>2$, the $T$-space 
 generated by any subset $\{\,x_1^px_2^p\cdots x_{i_1}^p, x_1^px_2^p\cdots x_{i_2}^p,\ldots\,\}$, $i_1<i_2<i_3<\cdots$,
 of $G$ has a $T$-space basis of size at most $i_2-i_1+1$.
\end{abstract}

\section{Introduction}
 In \cite{Gr} (and later in \cite{ShGr}, the survey paper with V. V. Shchigolev), A. V. Grishin proved that in 
 the free associative algebra with countably infinite generating set $\set x_1,x_2,\ldots\endset$ over an infinite field 
 of characteristic 2, the $T$-space generated by the set $\set x_1^2,x_1^2x_2^2,\ldots\endset$ 
 is not finitely based, and he raised the question as to 
 whether or not over a field
 of characteristic $p>2$, the $T$-space generated by $\set x_1^p,x_1^px_2^p,\ldots\endset$ is finitely based.
 This was resolved by V. V. Shchigolev in \cite{Sh}, wherein he proved that over an infinite field
 of characteristic $p>2$, this $T$-space is finitely based. Shchigolev then raised the question in \cite{Sh}
 as to whether every infinite subset of $\set x_1^p,x_1^px_2^p,\ldots\endset$ generates a finitely based
 $T$-space. In this paper, we prove that over an arbitrary field of characteristic $p>2$, every subset of 
 $\set x_1^p,x_1^px_2^p,\ldots\endset$ generates a $T$-space that can be generated as a $T$-space by finitely 
 many elements, and we give an upper bound for the size of a minimal generating set.

 Let $p$ be a prime (not necessarily greater than 2) and let $k$ denote an
 arbitrary field of characteristic $p$. Let $X=\set x_1,x_2,\ldots\endset$ be a countably infinite set, and 
 let $\kzerox$ denote the free associative $k$-algebra over the set $X$.
 
 \begin{definition}\label{definition: nsd}
  For any positive integer $d$, let 
  $$
   S^{(d)}= S^{(d)}(x_1,x_2,\ldots,x_d)=\sum_{\sigma\in \Sigma_d} \prod_{i=1}^d x_{\sigma(i)},
  $$
  where $\Sigma_d$ is the symmetric group on $d$ letters. Then define
  $\ns{1}{d}=\set S^{(d)}\endset^S$, the $T$-space generated by $\set S^{(d)}\endset$, and for all 
  $n\ge1$, $\ns{n+1}{d}=(\ns{n}{d}\ns{1}{d})^S$.
 \end{definition}
  
 Let $I:i_1<i_2<\cdots$ be a sequence of positive integers (finite or infinite), and then for each $n\ge1$, let
 $\rd{n,I}{d}=\sum_{j=1}^n \ns{i_j}{d}$. When the sequence $I$ is understood, we shall usually write $\rd{n}{d}$
 instead of $\rd{n,I}{d}$. Finally, let $\rd{\infty,I}{d}$ (even if the sequence is finite) denote 
 the $T$-space generated by $\{\, \ns{i}{d}\mid i\in I\,\}$. We shall prove that $\rd{\infty,I}{d}$
 has a $T$-space basis of size at most $i_2-i_1+1$.
 
 \begin{definition}\label{definition: nh}
  Let $\nh{1}=\set x_1^p\endset^S$, and for each $n\ge1$, let
  $\nh{n+1}=(\nh{n}\nh{1})^S$.
 \end{definition}
 
 Then for any positive integer $n$, let $L_{n,I}=\sum_{j=1}^n H_{i_j}$, and let $L_{\infty,I}$ denote the $T$-space generated
 by $\{\,h_i\mid i\in I\,\}$. We prove that $L_{\infty,I}$ is finitely generated as a $T$-space, with a $T$-space
 basis of size at most $i_2-i_1+1$. In particular, this proves that Shchigolev's conjecture is valid. 
 
 \section{Preliminaries}\label{section: fundamental results}
  In this section, $k$ denotes an arbitrary field of characteristic an arbitrary prime $p$, and $V_i$,
  $i\ge1$, denotes a sequence of $T$-spaces of $\kzerox$  satisfying the following two properties:
   \begin{list}{(\roman{sarparts})}{\usecounter{sarparts}}
    \item $(V_iV_j)^S=V_{i+j}$;
    \item for all $m\ge 1$, $V_{2m+1}\subseteq V_{m+1}+V_1$.
   \end{list}

\begin{lemma}\label{lemma: basic v result}
 For any integers $r$ and $s$ with $0<r<s$, $V_{s+t(s-r)}\subseteq V_r+V_s$ for all $t\ge0$.
\end{lemma}

\begin{proof}
 The proof is by induction on $t$. There is nothing to show for $t=0$. For $t=1$, let $m=s-r$ in (ii) to obtain
 that $V_{2s-2r+1}\subseteq V_{s-r+1}+V_1$, then multiply by $V_{r-1}$ to obtain
 $V_{r-1}V_{2s-2r+1}\subseteq V_{r-1}V_{s-r+1}+V_{r-1}V_1\subseteq (V_{r-1}V_{s-r+1})^S+(V_{r-1}V_1)^S=
 V_s+V_r$. But then $V_{2s-r}=(V_{r-1}V_{2s-2r+1})^S\subseteq V_s+V_r$, as required.
 
 Suppose now that $t\ge1$ is such that the result holds. Then $V_{s+(t+1)(s-r)}=(V_{s+t(s-r)}V_{s-r})^S
 \subseteq ((V_s+V_r)V_{s-r})^S=V_{2s-r}+V_s\subseteq V_r+V_s+V_s=V_r+V_s$. The result follows now by 
 induction.
\end{proof}

 For any increasing sequence $I: i_1<i_2<\cdots$ of positive integers, we shall refer to $i_2-i_1$ as the
 initial gap of $I$.

 \begin{proposition}\label{proposition: basic sequence}
  For any increasing sequence $I=\set i_j\endset_{j\ge1}$ of positive integers, there exists a set
  $J$ of size at most $i_2-i_1+1$ with entries positive integers
  such that the following hold:
    \begin{list}{(\roman{sarparts})}{\usecounter{sarparts}}
     \item $1,2\in J$;
     \item $\sum_{j=1}^\infty V_{i_j}=\sum_{j\in J} V_{i_j}$.
     \end{list}    
 \end{proposition}

\begin{proof}
  The proof of the proposition shall be by induction on the initial gap. By Lemma
  \ref{lemma: basic v result}, for a sequence with initial gap 1, we may take
  $J=\set i_1,i_2\endset$ . Suppose now that $l>1$ is an integer for which the result holds for all
  increasing sequences with initial gap less than $l$, and let $i_1<i_2<\cdots$ be a sequence with
  initial gap $i_2-i_1=l$. If for all $j\ge3$, $V_{i_j}\subseteq V_{i_1}+V_{i_2}$, then
  $J=\set 1,2\endset$ meets the requirements, so we may suppose that there exists $j\ge3$
  such that $V_{i_j}$ is not contained in $V_{i_1}+V_{i_2}$. By Lemma \ref{lemma: basic v result},
  this means that there exists $j\ge 3$ such that $i_j\notin \set i_2+ql\mid q\ge0\endset$. Let $r$ be 
  least such that $i_r\notin \set i_2+ql\mid q\ge0\endset$, so that there exists $t$ such that $i_2+tl<i_r<i_2+(t+1)l$.
  Form a sequence $I'$ from $I$ by first removing all entries of $I$ up to (but not including) $i_r$, then prepend
  the integer $i_2+tl$. Thus $i_1'$, the first entry of $I'$, is $i_2+tl$, while for all $j\ge2$, $i_j'=i_{r+j-2}$.
  Note that $i_2'-i_1'=i_r-(i_2+tl)\le l-1$.  By hypothesis, there exists a subset $J'$ of size at most 
  $i_2'-i_1'+1\le l=i_2-i_1$ that contains 1 and 2 and is such that
  $\sum_{j=1}^\infty V_{i_j'}=\sum_{j\in J'} V_{i_j'}$.  
  Set 
  $$
   J=\set 1,2\endset\cup\set r+j-2\mid j\in J',\ j\ge 2\endset.
  $$
  Then $|J|=|J'|+1\le i_2-i_1+1$ and
  $$
   V_{i_2+tl}+\sum_{j=r}^{\infty} V_{i_{j}} 
    =\sum_{j=1}^\infty V_{i_j'}=\sum_{j\in J'} V_{i_j'}= V_{i_2+tl}+\sum_{\substack{j\in J'\\j\ge2}} V_{i_j'}=
    V_{i_2+tl}+\sum_{\substack{j\in J\\j\ge 3}} V_{i_j}
   $$
   and by Lemma \ref{lemma: basic v result}, $V_{i_2+tl}\subseteq V_{i_2}+V_{i_2}$, so 
   \begin{align*}
    V_{i_1}+V_{i_2}+\sum_{j=r}^{\infty} V_{i_{j}}&=  V_{i_1}+V_{i_2}+ V_{i_2+tl}+\sum_{j=r}^{\infty} V_{i_{j}} 
    = V_{i_1}+V_{i_2}+ V_{i_2+tl}+\sum_{\substack{j\in J\\j\ge 3}} V_{i_j}\\
    &= V_{i_1}+V_{i_2}+ \sum_{\substack{j\in J\\j\ge 3}} V_{i_j}.
   \end{align*}
   Finally, the choice of $r$ implies that
   $$
    \sum_{j\in J} V_{i_j}=V_{i_1}+V_{i_2}+\sum_{\substack{j\in J\\j\ge 3}} V_{i_{j}}
    =V_{i_1}+V_{i_2}+\sum_{j=r}^{\infty} V_{i_{j}}=\sum_{j=1}^\infty V_{i_j}. 
  $$
  This completes the proof of the inductive step. 
\end{proof}

 We remark that in Proposition \ref{proposition: basic sequence}, it is possible to improve the bound from $i_2-i_1+1$
 to $2(\log_2(2(i_2-i_1))$.

 In the sections to come, we shall examine some important situations of the kind described above.
 
\section{The $\rd{n}{d}$ sequence}
  We shall have need of certain results that first appeared in \cite{CR}. For completeness, we include them
  with proofs where necessary. In this section, $p$ denotes an
  arbitrary prime, $k$ an arbitrary field of characteristic $p$, and $d$ an arbitrary positive integer.
     
  The proof of the first result is immediate.
  
  \begin{lemma}\label{lemma: basic}
    Let $d$ be a positive integer. Then 
    \begin{align*}
     S^{(d+1)}(x_1,x_2,\ldots,x_{d+1})&=\sum_{i=1}^{d+1}S^{(d)}(x_1,x_2,\ldots,\hat{x}_i,\ldots,x_{d+1})x_i\tag{1}\\ 
     &\hskip-70pt=S^{(d)}(x_1,x_2,\ldots,x_d)x_{d+1}+\sum_{i=1}^d S^{(d)}(x_1,x_2,\ldots,x_{d+1}x_i,\ldots,x_d)\tag{2}\\
     &\hskip-70pt=x_{d+1}S^{(d)}(x_1,x_2,\ldots,x_d)+\sum_{i=1}^d S^{(d)}(x_1,x_2,\ldots,x_ix_{d+1},\ldots,x_d).\tag{3}
    \end{align*}
  \end{lemma}
  
  \begin{corollary}\label{corollary: mod r1}
   Let $d$ be any positive integer. Then modulo $\ns{1}{d}$,
   $$
    S^{(d+1)}(x_1,x_2,\ldots,x_{d+1})\cong S^{(d)}(x_1,\ldots,x_d)x_{d+1}\cong x_{d+1}S^{(d)}(x_1,\ldots,x_d).
   $$
  \end{corollary}
  
  \proof
   This is immediate from (2) and (3) of Lemma \ref{lemma: basic}.
  \endproof

  We remark that Corollary \ref{corollary: mod r1} implies that for every $u\in S^{(d)}_1$ and $v\in \kzerox$,
  $\com u,v\in S^{(d)} _1$. While we shall not have need of this fact, we note that in \cite{Sh}, Shchigolev 
  proves that if the field is infinite, then for any $T$-space $V$, if $v\in V$, then $\com v,u\in V$ for any 
  $u\in \kzerox$.
  
 The next proposition is a strengthened version of Proposition 2.1 of \cite{CR}.
  
 \begin{proposition}\label{proposition: generalized r2=r3}
  For any $u,v\in \kzerox$, 
   \begin{list}{(\roman{sarparts})}{\usecounter{sarparts}}
   \item $(\ns{1}{d}uv)^S\subseteq \ns{1}{d}+(\ns{1}{d}u)^S+(\ns{1}{d}v)^S$; and
   \item $(uv\ns{1}{d})^S\subseteq \ns{1}{d}+(u\ns{1}{d})^S+(v\ns{1}{d})^S$.
   \end{list} 
 \end{proposition}

 \proof
  We shall prove the first statement; the proof of the second is similar and will be omitted.
  By (1) of Lemma \ref{lemma: basic},
  $$
    \sum_{i=1}^{d}S^{(d)}(x_1,\ldots,\hat{x}_i,\ldots,x_{d+1})x_i = S^{(d+1)}(x_1,\ldots,x_{d+1})- S^{(d)}(x_1,\ldots,x_d)x_{d+1}
  $$
  and by (2) of Lemma \ref{lemma: basic}, $S^{(d+1)}(x_1,\ldots,x_{d+1})- S^{(d)}(x_1,\ldots,x_d)x_{d+1}\in \ns{1}{d}$.
  Let $v\in \kzerox$. Then 
   \begin{align*}
  S^{(d)}(x_2,\ldots,x_{d+1})x_1v + &\sum_{i=2}^{d} S^{(d)}(x_1,\ldots,\hat{x}_i,\ldots,x_{d+1})x_iv\\
  &\hskip4pt=\sum_{i=1}^{d}S^{(d)}(x_1,x_2,\ldots,\hat{x}_i,\ldots,x_{d+1})x_iv\in (\ns{1}{d}v)^S\mkern-5mu\hbox to 0pt{.\hss}
  \end{align*}
  Now for each $i=2,\ldots,d$, we use two applications of Corollary \ref{corollary: mod r1} to obtain
  \begin{align*}
  S^{(d)}(x_1,\ldots,\hat{x}_i,\ldots,x_{d+1})x_iv &\cong S^{(d+1)}(x_1,\ldots,\hat{x}_i,\ldots,x_{d+1},x_iv)\\
  &\cong S^{(d)}(x_2,\ldots,\hat{x}_i,\ldots,x_{d+1},x_iv)x_1\mod{\ns{1}{d}}.
  \end{align*}
  Thus 
  {\abovedisplayskip=0pt\belowdisplayskip=4pt
  $$
  S^{(d)}(x_2,\ldots,x_{d+1})x_1v+\bigl((\sum_{i=2}^{d} S^{(d)}(x_2,\ldots,\hat{x}_i,\ldots,x_{i}v)\bigr)x_1\in (\ns{1}{d}v)^S+\ns{1}{d}.
  $$
  }
  Thus for $u\in\kzerox$, we obtain
  $S^{(d)}(x_2,\ldots,x_{d+1})uv \in (\ns{1}{d}u)^S+(\ns{1}{d}v)^S+\ns{1}{d}$,
  and so 
  $$(\ns{1}{d}uv)^S\subseteq (\ns{1}{d}u)^S+(\ns{1}{d}v)^S+\ns{1}{d},
  $$
  as required.
 \endproof 

\begin{corollary}\label{corollary: r result}
 Let $d$ be any positive integer. Then the sequence $\ns{n}{d}$, $n\ge 1$, satisfies
  \begin{list}{(\roman{sarparts})}{\usecounter{sarparts}}
  \item For all $m,n\ge 1$, $(\ns{m}{d}\ns{n}{d})^S=\ns{m+n}{d}$;
  \item For all $m\ge1$, $\ns{2m+1}{d}\subseteq \ns{m+1}{d}+\ns{1}{d}$.
  \end{list} 
\end{corollary}

\begin{proof}
 The first statement follows immediately from Definition \ref{definition: nsd} by an elementary
 induction argument. For the second statement, let $m\ge1$. Then by Proposition \ref{proposition: generalized r2=r3},
 for any $u,v\in\ns{m}{d}$, $(\ns{1}{d}uv)^S\subseteq \ns{1}{d}+(\ns{1}{d}u)^S+(\ns{1}{d}v)^S$,
 which implies that $(\ns{1}{d}\ns{m}{d}\ns{m}{d})^S\subseteq \ns{1}{d}+(\ns{1}{d}\ns{m}{d})^S$.
 By (i), this yields $\ns{2m+1}{d}\subseteq \ns{1}{d}+\ns{m+1}{d}$, as required.
\end{proof} 
 
\begin{theorem}\label{theorem: r}
 Let $I$ denote any increasing sequence of positive integers with initial gap $g$. Then $\rd{\infty,I}{d}$
 is finitely based, with a $T$-space basis of size at most $g+1$.
\end{theorem}

\begin{proof}
 Denote the entries of $I$ in increasing order by $i_j$, $j\ge1$.
 By Corollary \ref{corollary: r result} and Proposition \ref{proposition: basic sequence}, there exists a set $J$ of 
 positive integers with $|J|\le i_2-i_1+1$ and $\rd{\infty,I}{d}=\rd{n,I}{d}=
 \sum_{j\in J} \ns{i_j}{d}$. Since for each $i$, the $T$-space $\ns{i}{d}$ has a basis consisting of a single
 element, the result follows.
\end{proof}

%%------------------------------------  
\section{The $L_{n}$ sequence}
   
We shall make use of the following well known result. An element $u\in \kzerox$ is said to be
essential if $u$ is a linear combination of monomials with the property that each variable that appears
in any monomial appears in every monomial.

\begin{lemma}\label{lemma: essential}
 Let $V$ be a $T$-space and let $f\in V$. If $f=\sum f_i$ denotes the decomposition of $f$ into its
 essential components, then $f_i\in V$ for every $i$.
\end{lemma}

\begin{proof}
 We induct on the number of essential components, with obvious base case. Suppose that
 $n>1$ is an integer such that if $f\in V$ has fewer than $n$ essential components, then each belongs to $V$,
 and let $f\in V$ have $n$ essential components. Since $n>1$, there is a variable $x$ that appears in some but not
 all essential components of $f$. Let $z_x$ and $f_x$ denote the sum of the essential components of $f$ in which $x$ appears,
 respectively,  does not appear. Then evaluate at $x=0$ to obtain that $f_x=f\mkern-5mu\mid_{_{x=0}}\in V$, and thus $z_x=f-f_x\in V$
 as well. By hypothesis, each essential component of $f_x$ and of $z_x$ belongs to $V$, and thus every essential component
 of $f$ belongs to $V$, as required. 
\end{proof}

\begin{corollary}\label{corollary: sp in h1}
$\nsp{1}\subseteq \nh{1}$.
\end{corollary}

\begin{proof}
 $S^{(p)}$ is one of the essential components of $(x_1+x_2+\cdots+x_p)^p$, and since $(x_1+x_2+\cdots+x_p)^p
 \in \nh{1}$, it follows from Lemma \ref{lemma: essential} that $S^{(p)}\in\nh{1}$. Thus $\nsp{1}\subseteq \nh{1}$.
\end{proof}

\begin{corollary}\label{corollary: ns in h}
 For every $m\ge1$, $\nsp{m}\subseteq \nh{m}$.
\end{corollary}

\begin{proof}
 The proof is an elementary induction, with Corollary \ref{corollary: sp in h1} providing the base case.
\end{proof}

\begin{corollary}\label{corollary: closed under commutator}
 For any $u\in \nh{1}$ and any $v\in\kzerox$, $\com u,v\in\nh{1}$.
\end{corollary} 

\begin{proof}
 It suffices to observe that 
 {\abovedisplayskip=2pt\belowdisplayskip=4pt
 $$
  \com x^p,v=\sum_{i=0}^p x^i\com x,v x^{p-i} = \frac{1}{(p-1)!}S^{(p)}(x,x,\ldots,x,\com x,v),
 $$
 }
 which belongs to $\nh{1}$ by virtue of Corollary \ref{corollary: sp in h1}.
\end{proof}

We remark again that in \cite{ShGr}, Shchigolev proves that if $k$ is infinite, then every $T$-space in $\kzerox$ is
closed under commutator in the sense of Corollary \ref{corollary: closed under commutator}. Since we have not
required that $k$ be infinite, we have provided this closure result (see also Lemma \ref{lemma: h commutator} below).
 
\begin{lemma}\label{lemma: basic h result}
 For any $m,n\ge1$, $(\nh{m}\nh{n})^S=\nh{m+n}$.
\end{lemma}

\begin{proof}
 The proof is by an elementary induction on $n$, with Definition \ref{definition: nh} providing the base case.
\end{proof}
 
\begin{lemma}\label{lemma: h base case}
For any $m\ge1$, $(\nsp{1}\nh{2m})^S\subseteq \nh{1}+\nh{m+1}$ and $(\nh{2m}\nsp{1})^S\subseteq \nh{1}+\nh{m+1}$.
\end{lemma}

\begin{proof}
By Proposition \ref{proposition: generalized r2=r3} (i), for any $u,v\in \nh{m}$, we have
$\nsp{1}uv\subseteq \nsp{1}+(\nsp{1}u)^S+(\nsp{1}v)^S$. By Corollary \ref{corollary: ns in h}, this gives
$\nsp{1}\nh{m}\nh{m}\subseteq \nh{1}+(\nh{1}\nh{m})^S$, and then from Lemma \ref{lemma: basic h result}, we 
obtain $\nsp{1}\nh{2m}\subseteq \nh{1}+\nh{m+1}$. The proof of the second part is similar.
\end{proof}

\begin{lemma}\label{lemma: h commutator}
 Let $m\ge1$. For every $u\in \nh{m}$ and $v\in \kzerox$, $\com u,v\in \nh{m}$.
\end{lemma}

\begin{proof}
 The proof is by induction on $m$, with Corollary \ref{corollary: closed under commutator} providing the
 base case. Suppose that $m\ge 1$ is such that the result holds. It suffices to prove that for any
 $v\in \kzerox$, $\com x_1^px_2^p\cdots x_m^px_{m+1}^p,v\in \nh{m+1}$. We have
 $$
  \com x_1^px_2^p\cdots x_m^px_{m+1}^p,v=\com x_1^px_2^p\cdots x_m^p,v x_{m+1}^p+ x_1^px_2^p\cdots x_m^p\com x_{m+1}^p,v.
 $$ 
 By hypothesis, $\com x_1^px_2^p\cdots x_m^p,v\in \nh{m}$, while $x_{m+1}^p\in \nh{1}$ and thus by 
 Corollary \ref{corollary: closed under commutator}, $\com x_{m+1}^p,v\in \nh{1}$ as well. Now by definition, 
 $\com x_1^px_2^p\cdots x_m^p,v x_{m+1}^p\in \nh{m+1}$ and 
 $x_1^px_2^p\cdots x_m^p\com x_{m+1}^p,v\in \nh{m+1}$, which completes the proof of the inductive step.
\end{proof}

\begin{lemma}\label{fundamental h result}
 Let $m\ge1$. Then $\nh{i}S^{(p)}\nh{2m-i}\subseteq \nh{1}+\nh{m+1}$ for all $i$ with $1\le i\le 2m-1$. 
\end{lemma}

\begin{proof}
 Let $m\ge1$. We consider two cases: $2m-i\ge m$ and $2m-i<m$. Suppose that 
 $2m-i\ge m$, and let $u\in \nh{i}$, $w\in \nh{m-1}$ and $z\in \nh{m-i+1}$. Then $uS^{(p)}wz = (\com u,{S^{(p)}w}+S^{(p)}wu)z=
 \com u,{S^{(p)}w}z+S^{(p)}wuz$.  
  Since $u\in\nh{i}$, it follows from Lemma \ref{lemma: h commutator} that $\com u,{S^{(p)}w}\in \nh{i}$.
  But then by Lemma \ref{lemma: basic h result}, $\com u,{S^{(p)}w}z\in \nh{i+m-i+1}=\nh{m+1}$. As well, by
  Corollary \ref{corollary: sp in h1} and Lemma \ref{lemma: basic h result}, 
  $S^{(p)}wuz\in \nsp{1}\nh{m-1+i+m-i+1}=\nsp{1}\nh{2m}$, and by Lemma \ref{lemma: h base case}, $\nsp{1}\nh{2m}\subseteq
  \nh{1}+\nh{m+1}$. Thus $uS^{(p)}wz\in\nh{1}+\nh{m+1}$. This proves that $\nh{i}S^{(p)}\nh{m-1}\nh{m-i+1}\subseteq
  \nh{1}+\nh{m+1}$, and so by Lemma \ref{lemma: basic h result}, $\nh{i}S^{(p)}\nh{2m-i}=\nh{i}S^{(p)}(\nh{m-1}\nh{m-i+1})^S
  \subseteq \nh{1}+\nh{m+1}$. The argument for the case when $2m-i< m$ is similar and is therefore omitted.
\end{proof}
 
\begin{proposition}\label{proposition: h result}
 Let $p>2$. Then for every $m\ge 1$, $\nh{2m+1}\subseteq \nh{1}+\nh{m+1}$.
\end{proposition}

\begin{proof}
  First, consider the expansion of $(x+y)^{p}$ for any $x,y\in \kzerox$. It will be convenient to introduce the following 
  notation. Let $J_p=\set 1,2,\ldots,p\endset$. For any $J\subseteq J_p$, let $P_J=\prod_{i=1}^p z_i$,
  where for each $i$, $z_i=x$ if $i\in J$, otherwise $z_i=y$. As well, for each $i$ with $1\le i\le p-1$, we
  shall let $S^{(p)}(x,y;i)=S^{(p)}(\underbrace{x,x,\ldots,x}_{i},\underbrace{y,y,\ldots,y}_{p-i})$. Observe
  that $S^{(p)}(x,y;i)=i!(p-i)!\sum_{\substack{J\subseteq J_p\\|J|=i}} P_J$. We have
  {\abovedisplayskip=-1pt\belowdisplayskip=4pt
  $$
     (x+y)^{p}= \sum_{i=0}^p\sum_{\substack{J\subseteq J_p\\ |J|=i}}P_J
  =y^p+x^p+\sum_{i=1}^{p-1}\frac{1}{i!(p-i)!}S^{(p)}(x,y;i).
  $$
  }
  Let $u=\sum_{i=1}^{p-1}\frac{1}{i!(p-i)!}S^{(p)}(x,y;i)$, so that
  $(x+y)^p=x^p+y^p +u$, and note that $u\in \nsp{1}$. Then
  $$
   (x+y)^{2p}=y^{2p}+x^{2p}+2x^py^p+\com y^p,{x^p}+u^2+(x^p+y^p)u+u(x^p+y^p).
  $$
  Since $(x+y)^{2p}$, $x^{2p}$, $y^{2p}$, and, by Lemma \ref{lemma: h commutator}, $\com y^p,{x^p}$ all
  belong to $\nh{1}$, it follows (making use of Corollary \ref{corollary: ns in h} where necessary) that 
  $2x^py^p\in \nh{1}+\nh{1}\nsp{1}+\nsp{1}\nh{1}$.

  Consequently, for any $m\ge 1$, 
  $$
   x_1^p\prod_{i=1}^m(2x_{2i}^px_{2i+1}^p)\in
   \nh{1}(\nh{1}+\nh{1}\nsp{1}+\nsp{1}\nh{1})^m.
  $$
  By Corollary \ref{corollary: sp in h1}, Lemma \ref{lemma: basic h result},
  and Lemma \ref{fundamental h result}, $\nh{1}(\nh{1}+\nh{1}\nsp{1}+\nsp{1}\nh{1})^m
  \subseteq \nh{1}+\nh{m+1}$, and since $p>2$, it follows that $\prod_{i=1}^{2m+1} x_i^p\in \nh{1}+\nh{m+1}$.
  Thus $\nh{2m+1}\subseteq \nh{1}+\nh{m+1}$, as required.
\end{proof}

\begin{theorem}[Shchigolev's conjecture]\label{theorem: shchigolevs conjecture} 
 Let $p>2$ be a prime and $k$ a field of characteristic $p$. For any increasing sequence 
 $I=\set i_j\endset_{j\ge1}$, $\nl{\infty,I}$ is a finitely based $T$-space of $\kzerox$, with
 a $T$-space basis of size at most $i_2-i_1+1$.
\end{theorem}

\begin{proof}
 By Lemma \ref{lemma: basic h result} and Proposition \ref{proposition: h result},
 the sequence $\nh{n}$ of $T$-spaces of $\kzerox$ meets the requirements of Section \ref{section: fundamental results}.
 Thus by Proposition \ref{proposition: basic sequence}, for any increasing sequence $I=\set i_j\endset_{j\ge1}$ of positive 
 integers, there exists a set $J$ of positive integers such that $|J|\le i_2-i_1+1$
 and $\nl{\infty,I}=\sum_{j=1}^\infty \nh{i_j}=\sum_{j\in J} \nh{i_j}$.
 Since for each $i$, $\nh{i}$ has $T$-space basis $\set x_1^px_2^p\cdots x_i^p\endset$, it follows that $\nl{\infty,I}$
 has a $T$-space basis of size $|J|\le i_2-i_1+1$.
\end{proof}

Shchigolev's original result was that for the sequence $I^+$ of all positive integers, $\nl{\infty,I^+}$ is a finitely-based 
$T$-space, with a $T$-space basis of size at most $p$. It was then shown in \cite{CR}, a precursor to this work,
that $\nl{\infty,I^+}$ has in fact a $T$-space basis of size at most 2 (the bound of Theorem \ref{theorem: shchigolevs conjecture},
since $i_1=1$ and $i_2=2$).
 
It is also interesting to note that the results in this paper apply to finite sequences. Of course, if $I$ is a
finite increasing sequence of positive integers, then $\nl{\infty,I}$ has a finite $T$-space basis, but by the preceding work,
we know that it has a $T$-space basis of size at most $i_2-i_1+1$.


\begin{thebibliography}{0}

\bibitem{CR} C. Bekh-Ochir and S. A. Rankin, {\em On a problem of A. V. Grishin}, preprint, arXiv:0909.2266.

\bibitem{Gr} A. V. Grishin, {\em T-spaces with an infinite basis over a field of characteristic 2}, 
International Conference in Algebra and Analysis Commemorating the Hundredth Anniversary of N. G. 
Chebotarev, Proceedings, Kazan 5--11, June, 1994, p. 29 (Russian).

\bibitem{ShGr} A. V. Grishin and V. V. Shchigolev, {\em T-spaces and their applications}, Journal
of Mathematical Sciences, Vol 134, No. 1, 2006, 1799--1878 (translated from Sovremennaya Matematika i-Ee Prilozheniya,
Vol. 18, Alg\`ebra, 2004).

\bibitem{Sh} V. V. Shchigolev, {\em Examples of $T$-spaces with an infinite basis}, Sbornik 
Mathematics, Vol 191, No. 3, 2000, 459--476.
\end{thebibliography}
\end{document}